\documentclass[11pt, oneside]{article} 

\usepackage{geometry}                		
\usepackage{amsmath}
\usepackage{amsthm}
\usepackage{mathptmx}
\usepackage{graphicx}				
\usepackage{mathrsfs}					
\usepackage{amssymb}
\usepackage{enumerate}
\usepackage{authblk}

\theoremstyle{plain}
\newtheorem{theorem}{Theorem}
\newtheorem{lemma}{Lemma}
\newtheorem{notation}{Notation}
\newtheorem{prop}{Proposition}
\newtheorem{remark}{Remark}%
\theoremstyle{definition}
\newtheorem{definition}{Definition}%
\newtheorem{assumption}{Assumption}
\providecommand{\keywords}[1]
{
\small
\textbf{\textit{Keywords---}} #1
}
  \def\hoge<#1>{\langle #1 \rangle}
  
\DeclareMathOperator*{\argmax}{arg\,max}

\title{The maximum likelihood type estimator of SDEs with fractional Brownian motion under small noise asymptotics in the rough case}

\author{Shohei Nakajima \thanks{email:~snakajima@aoni.waseda.jp}}
\affil{School of International Liberal Studies, Waseda University\\1-6-1, Nishi-waseda, Shinjuku-ku, Tokyo, 169-8050, Japan}

\begin{document}
\maketitle
\begin{abstract}
  We study the problem of parametric estimation for continuously observed stochastic differential equation driven by fractional Brownian motion with Hurst index $H\in(1/3,1/2)$. Under some assumptions on drift and diffusion coefficients, we construct maximum likelihood estimator and establish its the asymptotic normality and moment convergence of the drift parameter when a small dispersion coefficient $\varepsilon\rightarrow 0$.
\end{abstract}
 \section{Introduction}
 Let $\{X_t^\varepsilon\}_{t\in[0,T]}$ be the solution to the following d-dimensional stochastic differential equations:
  \begin{equation}
   \label{SDE}
   X^\varepsilon_t=X_0+\int_0^tb(X_s^\varepsilon,\theta_0)ds+\varepsilon \int_0^t\sigma(X_s^\varepsilon)dB_s,~~~t\in[0,T],
 \end{equation}
where $\theta_0\in\Theta$ with $\Theta$ being a bounded and open convex subset of $\mathbb{R}^m$ admitting Sobolev's inequalities for embedding $W^{1,q}(\Theta)\hookrightarrow C(\bar{\Theta})$.
 Futhermore, $X_0\in\mathbb{R}^d,~\varepsilon>0$ are known constants, $b$ is an $\mathbb{R}^d$ valued function defined on $\mathbb{R}^d\times\Theta$, $\sigma$ is an $\mathbb{R}^d\otimes\mathbb{R}^r$ valued function defined on $\mathbb{R}^d$,
  $B_t=\left(B_t^{(1)},\dots,B_t^{(r)}\right),~B_t^{(i)},~i=1,\dots,r$ are independent one--dimensional fractional Brownian motion with Hurst index $H=(H_1,\dots,H_r)$, $H_i\in(\frac{1}{3},\frac{1}{2})$ and the stochastic integral appearing in \eqref{SDE} is defined by the rough integral sense (see Definition \ref{rough integral}).
  It is known that the fractional Brownian motion is a semi-martingale if and only if its Hurst index is equal to $1/2$. Therefore, the tool of the stochastic analysis for semimartingale is not applicable for equation \eqref{SDE}. The theory of rough path analysis allows us to study \eqref{SDE} in the pathwise sense and guarantees the existence and uniqueness of a solution. The reader is referred to \cite{Fritz}, \cite{Gubinelli} for further details.

  Recently, the stochastic differential equation of the type \eqref{SDE} is more important in the field of the mathematical finance (see \cite{Gatheral}). However, the statistical theory on the equation \eqref{SDE} is not yet sufficiently developed to be applied in practice. Our purpose is likelihood analysis of the equation \eqref{SDE}, the most fundamental aspect in the statistical inference. We construct the maximum likelihood type estimator for unknown parameter $\theta_0\in\Theta$ from a realization $\{X_t^\varepsilon\}_{t\in[0,T]}$ and derive its asymptotic properties when $\varepsilon\rightarrow 0.$

  In the case where $H_i=1/2$, for all $i=1,\dots,r$, that is, $\{B_t\}_{t\in[0,T]}$ is an $r$-dimensional Brownian motion, parameter estimation problems have been studied by many authors. In particular, the maximum likelihood estimator via the likelihood function based on the Girsanov density is the one of the optimal methods for estimation (see Prakasa Rao \cite{Rao}, Liptser and Shiryaev \cite{Liptser} and Kutoyants \cite{Kutoyants}).

  The parametric inference for stochastic differential equations driven by the fractional Brownian motion has been studied by Brouste and Kleptsyna \cite{Brouste}, Kleptsyna and Le Breton \cite{Kleptsyna} and Tudor and Viens \cite{Tudor} under the assumption $T\rightarrow\infty$. In all these papers, the drift functions depend linearly on the parameter $\theta\in\Theta$ and the maximum likelihood estimators have an explicit expression. When the drift function is nonlinear in $\theta\in\Theta$, Chiba \cite{Chiba} proposed an M-estimator based on the likelihood function and established asymptotic properties of the estimator when the Hurst index $H$ is contained in $(1/4,1/2)$.
  For the least squares type estimators, the reader may refer to  Hu and Nualart \cite{Hu}, Neuenkirch and Tindel \cite{Neuenkirch}, Hu et. al \cite{Hu2} and Marie \cite{Marie}. Papavasiliou and Ladroue \cite{Papavasiliou} studied the statistical inference for differential equations driven by rough paths. They proposed an estimator based on the method of moments under the assumption which is many independent observations of sample paths over a short time interval being available.

   The maximum likelihood type estimator for one-dimensional stochastic  differential equations with small fractional noise have been studied by Nakajima and Shimizu \cite{Nakajima} in additive noise case, that is $\sigma(x)=1$ for $x\in\mathbb{R}$. Moreover, the multiplicative noise case with $H>1/2$ is studied in \cite{mnakajima}. In the case of our model, the construction of the maximum likelihood estimator is more complicated because of the need to deal with enhanced processes when deriving the likelihood function. The main focus of this paper is on the construction of the estimator and the derivation of the apriori estimate for the solution of the equation \eqref{SDE} to prove asymptotic properties of the estimator. The proof of the asymptotic properties of our estimator is almost same as in \cite{Nakajima}.

  This paper organized as follows: in Section 2 we make some notations and introduce the framework of the rough path analysis. In Section 3 we derive the apriori estimate for the solution of the equation \eqref{SDE}. In Section 4 we construct the maximum likelihood type estimator and state its asymptotic properties. In Section 5 we prove one of the sufficient conditions of the polynomial type large deviation inequality investigated by Yoshida \cite{Yoshida}.
  \section{Preliminaries}
  Without loss of generality, we assume that $\varepsilon\in(0,1]$. Let $(\Omega,\mathcal{F},P)$ be a probability space and $\mathcal{F}_t$ is the filtration generated by $\{B_s\}_{s\in[0,t]}$. Let us recall necessary facts from rough path theory of integration (for further details, see \cite{Fritz}). Suppose that $V,W$ are finite-dimensional Banach spaces and $\mathcal{L}(V,W)$ be the space of linear bounded operators from $V$ to $W$. The norm in any finite dimensional Banach space will be denoted by $|\cdot|$.
 For $0<\alpha\le 1$, we denote $\alpha$-H\"older continuous functions space by
\begin{equation*}
   C^\alpha([0,T],V):=\left\{Y\in C\left([0,T],V\right)|~\|Y\|_{\alpha,[0,T]}<\infty\right\}
\end{equation*}
where
\begin{equation*}
  \|Y\|_{\alpha,[0,T]}:=\sup_{s,t\in[0,T]:s\ne t}\frac{|Y_t-Y_s|}{|t-s|^\alpha}.
\end{equation*}
We define the space $C_2^\alpha([0,T]^2; V^{\otimes 2})$ which is the set of all two-variables functions $\mathbb{Y}:[0,T]^2\rightarrow V^{\otimes 2}$ such that the norm
\begin{equation*}
  \|\mathbb{Y}\|_{2\alpha,[0,T]}:=\sup_{s,t\in[0,T]: s\ne t}\frac{|\mathbb{Y}_{s,t}|}{|t-s|^{2\alpha}},
\end{equation*}
is finite.
\begin{notation}
  For the one-variable function $Y$ on $[0,T]$ we will use the notation $Y_{s,t}:=Y_t-Y_s$.
\end{notation}
\begin{definition}
We call $\alpha$--H\"older rough path in symbols $\mathscr{C}^\alpha([0,T],V)$, as those pairs $(Y,\mathbb{Y})$ such that $Y\in C^\alpha([0,T],V),~\mathbb{Y}\in C^\alpha_2([0,T]^2,V^{\otimes 2})$ and Chen's relations
\begin{equation*}
  \mathbb{Y}_{s,t}-\mathbb{Y}_{s,u}-\mathbb{Y}_{u,t}=Y_{s,u}\otimes Y_{u,t},
\end{equation*}
hold for each triple $(s,u,t)\in[0,T]^3$.
 \end{definition}
  For $(Y,\mathbb{Y})\in\mathscr{C}^\alpha([0,T],V)$, we denote the $\alpha$-H\"older rough path semi-norm by
 \begin{equation*}
   \|(Y,\mathbb{Y})\|_{\alpha,[0,T]}:=\|Y\|_{\alpha,[0,T]}+\|\mathbb{Y}\|^{1/2}_{\alpha,[0,T]}.
 \end{equation*}
 \begin{definition}
   Given a path $Y\in C^\alpha([0,T],V)$, we say that $Z\in C^\alpha([0,T],\mathcal{L}(V,W))$ is controlled by $Y$ if there exists $Z^\prime\in C^\alpha([0,T],\mathcal{L}\left(V,\mathcal{L}(V,W)\right))$ such that the relation
   \begin{equation*}
     Z_{s,t}=Z^{\prime}_sY_{s,t}+R^{Z}_{s,t},
   \end{equation*}
    holds with $R^Z\in C_2^\alpha([0,T],\mathcal{L}(V,W))$. The set of all pairs $(Z,Z^\prime)$ such that $Z$ is controlled by $Y$ is denoted by $\mathcal{D}^{2\alpha}_Y([0,T],\mathcal{L}(V,W))$.
 \end{definition}
 \begin{definition}
   \label{rough integral}
   Let $(Y,\mathbb{Y})\in\mathscr{C}^\alpha([0,T],V)$ and $(Z,Z^\prime)\in\mathcal{D}_Y^{2\alpha}([0,T],\mathcal{L}(V,W))$. We call rough integral of $Z$ against $(Y,\mathbb{Y})$ as
   \begin{equation*}
     \int_0^TZ_sdY_s:=\lim_{|\mathcal{P}|\rightarrow 0}\sum_{t_k,t_{k+1}\in\mathcal{P}}\left(Z_{t_k}Y_{t_k,t_{k+1}}+Z_{t_k}^\prime\mathbb{Y}_{t_k,t_{k+1}}\right),
   \end{equation*}
   where $\mathcal{P}$ denotes partition of $[0,T]$ and $|\mathcal{P}|$ is the length of the largest element.
 \end{definition}
\begin{prop}
  \label{rie}
  Let $(Y,\mathbb{Y})\in\mathscr{C}^\alpha([0,T],V)$ and $(Z,Z^\prime)\in\mathcal{D}_Y^{2\alpha}([0,T],\mathcal{L}(V,W))$. Then there exists a constant $C>0$ such that for any $s,t\in[0,T]$
  \begin{equation*}
    \left|\int_s^tZ_rdY_r-Z_sY_{s,t}-Z^\prime_s\mathbb{Y}_{s,t}\right|\le C\left(\|Y\|_\alpha\|R^Z\|_{2\alpha}+\|\mathbb{Y}\|_{2\alpha}\|Z^\prime\|_\alpha\right)|t-s|^{3\alpha}
  \end{equation*}
  holds.
\end{prop}
\begin{remark}
  \label{CI}
  For $(Y,\mathbb{Y})\in\mathscr{C}^\alpha([0,T],V)$, we consider two controlled rough paths $(Z,Z^\prime)\in\mathcal{D}_Y^{2\alpha}([0,T],W),~(G,G^\prime)\in\mathcal{D}_Y^{2\alpha}([0,T],\mathcal{L}(V,W))$. Then we can define the integral of $Z$ against $G$ by
  \begin{equation*}
    \int_s^tZ_udG_u:=\lim_{|\mathcal{P}|\rightarrow 0}\sum_{t_k,t_{k+1}\in\mathcal{P}}\left(Z_{t_k}G_{t_k,t_{k+1}}+Z_{t_k}^\prime G_{t_k}^\prime\mathbb{Y}_{t_k,t_{k+1}}\right),
  \end{equation*}
  and there exists $C>0$ such that the following inequality holds:
  \begin{equation*}
    \left|\int_s^tZ_udG_u-Z_sG_{s,t}-Z_{s}^\prime G_{s}^\prime\mathbb{Y}_{s,t}\right|\le C\left(\|Y\|_\alpha\|R^Z\|_{2\alpha}+\|\mathbb{Y}\|_{2\alpha}\|Z^\prime G^\prime\|_\alpha\right)|t-s|^{3\alpha}.
  \end{equation*}
\end{remark}

Let us turn back to equation \eqref{SDE} and consider the case $V=\mathbb{R}^r,~W=\mathbb{R}^d$. Define a two-variable function $\mathbb{B}$ as follows:
    \begin{center}
\begin{equation*}
  \begin{aligned}
      \mathbb{B}_{s,t}&=\left(\mathbb{B}_{s,t}^{(i,j)}\right)_{i,j=1}^r,\\
      \mathbb{B}_{s,t}^{(i,j)}&\overset{L^2}{=}\lim_{|\mathcal{P}|\rightarrow 0}\int_\mathcal{P}B_{s,r}^{(i)}dB_r^{(j)},~\int_\mathcal{P}B_{s,r}^{(i)}dB_r^{(j)}=\sum_{t_k,t_{k+1}\in\mathcal{P}}B_{s,t_k}^{(i)}B_{t_k,t_{k+1}}^{(j)},~~~
      1\le i<j\le r,\\
      \mathbb{B}^{(i,i)}_{s,t}&=\frac{1}{2}(B_{s,t}^{(i)})^2,~~~1\le i\le r,\\
      \mathbb{B}^{(i,j)}_{s,t}&=-\mathbb{B}^{(j,i)}_{s,t}+B_{s,t}^{(i)}B_{s,t}^{(j)},
  \end{aligned}
\end{equation*}
    \end{center}
    for $s,t\in[0,T]$, where $\mathcal{P}=\left\{s=t_0<t_1<\cdots<t_n=t\right\}$ is an arbitrary partition of the segment $[s,t]$. Note that the integral sums in $\mathbb{B}_{s,t}^{i,j}$ have a finite $L^2$-limit by Proposition 10.3 in \cite{Fritz}.

\begin{prop}
  \label{rfb}
  For any fixed $\alpha\in(1/3,h)$ such that $0<h<\min_{i=1,\dots,r}H_i$, the pair $(B,\mathbb{B})$ is geometirc rough path, that is, $(B,\mathbb{B})\in\mathscr{C}^\alpha([0,T],\mathbb{R}^r)$ and the following relation holds for $s,t\in[0,T]$,
    \begin{equation*}
      {\rm Sym}(\mathbb{B}_{s,t}):=\frac{1}{2}\left(\mathbb{B}_{s,t}+\mathbb{B}^*_{s,t}\right)=\frac{1}{2}B_{s,t}\otimes B_{s,t},
    \end{equation*}
where $*$ represents the transpose of a matrix. Moreover,
\begin{equation*}
  E\|\mathbb{B}\|_{2\alpha}^q<\infty,
\end{equation*}
   for any $q\ge 1$.
\end{prop}
\begin{remark}
  \label{SC}
  Let $(Z,Z^\prime)\in\mathcal{D}_{B}^{2\alpha}\left([0,T],\mathbb{R}^d\right)$ and $\phi\in C^2(\mathbb{R}^d,\mathbb{R}^d\otimes\mathbb{R}^r)$.
  Then one can define the controlled rough path $(\phi(Z),\phi(Z)^\prime)\in\mathcal{D}_{B}^{2\alpha}\left([0,T],\mathbb{R}^d\otimes\mathbb{R}^r\right)$ by
  \begin{equation*}
    \phi(Z)_t=\phi(Z_t),~\phi(Z_t)^\prime=\nabla\phi(Z_t)Z_t^\prime.
  \end{equation*}
\end{remark}
 \begin{definition}
   \label{rs}
   A process $\{X_t^\varepsilon\}_{t\in[0,T]}$ is called a solution to \eqref{SDE} if $(X^\varepsilon,\varepsilon\sigma(X^\varepsilon))\in\mathcal{D}_{B}^{2\alpha}\left([0,T],\mathbb{R}^d)\right)$ and the equality
   \begin{equation*}
     X_t^\varepsilon=X_0+\int_0^tb(X_s^\varepsilon,\theta_0)ds+\varepsilon\int_0^t\sigma(X_s^\varepsilon)dB_s,
   \end{equation*}
   holds a.s. where the integral with respect to $B$ appearing to right hand side is understood in the sense of Definition \ref{rough integral}.
 \end{definition}
 \section{Apriori estimate}
 Throughout this paper, we will use the following notations:
 \begin{notation}
 For any $a,b\ge 0$, the symbol $a\lesssim b$ means that there exists a universal constant $C>0$ such that $a\le Cb$.
 \end{notation}
 \begin{notation}
   Let $C^{k,l}(\mathbb{R}^d\times\Theta,\mathbb{R}^d)$ be a set of functions $f(x,\theta):\mathbb{R}^d\times\Theta\rightarrow\mathbb{R}^d$ that is $k$ and $l$ times differentiable with respect to $x$ and $\theta$, respectively.
   Moreover, let $C_b^m(\mathbb{R}^d,\mathbb{R}^d\otimes\mathbb{R}^r)$ be a set of functions $f(x):\mathbb{R}^d\rightarrow\mathbb{R}^d\otimes\mathbb{R}^r$ which are $m$ times differentiable and all derivatives up to $m$ times are bounded.
 \end{notation}
 In this section, we establish apriori estimate of the solution \eqref{SDE} in the sense of the Definition \ref{rs}. We introduce the following assumptions on the coefficients of the equation \eqref{SDE}.
 \begin{assumption}
   \label{(A1)}
   The function $b$ in \eqref{SDE} is of $C^{1,4}(\mathbb{R}^d\times\Theta,\mathbb{R}^d)$-class and there exist constants $c_1, c_2, N>0$ such that for every $x,y\in\mathbb{R}^d,~\theta\in\Theta$ and integers $0\le i\le 4$, $0\le j\le 1$,
 \begin{equation*}
   \begin{aligned}
     |b(x,\theta)-b(y,\theta)|\le c_1|x-y|,~~~\left|\nabla_x^j\nabla_\theta^ib(x,\theta)\right|\le c_2(1+|x|^N).
   \end{aligned}
 \end{equation*}
   \end{assumption}
   \begin{assumption}
     \label{(A2)}
     The function $\sigma$ in \eqref{SDE} is of $\sigma\in C^3_b(\mathbb{R}^d,\mathbb{R}^d\otimes\mathbb{R}^r)$-class.
   \end{assumption}
 In order to establish some estimates of the solution to \eqref{SDE}, we prepare the following lemma which is found in Exercise 4.24 in \cite{Fritz}.
 \begin{lemma}
   \label{HHE}
   Let $\alpha\in(0,1),~\delta>0,$ and $Z\in C^\alpha([0,T];V)$. Assume that
   \begin{equation*}
     \|Z\|_{\alpha,\delta}:=\sup\frac{|Z_{s,t}|}{|t-s|^\alpha}\le M,
   \end{equation*}
   where sup is restricted to time $s,t\in[0,T]$ and $|t-s|\le \delta$. Then
   \begin{equation*}
     \|Z\|_{\alpha,[0,T]}\le M(1\vee 2\delta^{-(1-\alpha)}).
   \end{equation*}
 \end{lemma}
 \begin{lemma}
   \label{HE}
   Under Assumptions \ref{(A1)} and \ref{(A2)}, for every $1/3<\alpha<h$, there exist constants $C_i>0,~i=1,2,3,$ depending on $T, X_0, \alpha$ such that
   \begin{align}
     &\|X^\varepsilon\|_{\alpha,[0,T]}\le C_1\left(1+\|(B,\mathbb{B})\|_{\alpha,[0,T]}^{1/\alpha}\right)\\
     \label{RLX}
     &\|R^{X^\varepsilon}\|_{2\alpha,[0,T]}\le C_2\left(1+\|(B,\mathbb{B})\|_{\alpha,[0,T]}^{2/\alpha}\right)\\
     \label{RLSX}
     &\|R^{\sigma(X^\varepsilon)}\|_{2\alpha,[0,T]}\le C_2\left(1+\|(B,\mathbb{B})\|_{\alpha,[0,T]}^{2/\alpha}\right).
   \end{align}
 \end{lemma}
 \begin{proof}
   Using Proposition \ref{rie} and Assumption \ref{(A1)},\ref{(A2)}, for $s,t\in[0,T]$,
\begin{equation*}
  \begin{aligned}
   & \left|R_{s,t}^{X^\varepsilon}\right|=\left|X^\varepsilon_{s,t}-\varepsilon\sigma(X^\varepsilon_s)B_{s,t}\right|
    \\
    &\le\left|\int_s^tb(X^\varepsilon_u,\theta_0)du\right|
+\varepsilon\left|\int_s^t\sigma(X_u^\varepsilon)dB_u-\sigma(X_s^\varepsilon)B_{s,t}-\varepsilon\nabla_x\sigma(X_s^\varepsilon) \sigma(X_s^\varepsilon)\mathbb{B}_{s,t}\right|\\
&+\varepsilon^2\left|\nabla_x\sigma(X_s^\varepsilon) \sigma(X_s^\varepsilon)\mathbb{B}_{s,t}\right|
\\
    &\lesssim\left(1+\|X^\varepsilon\|_{\infty,[0,T]}\right)|t-s|+\|\mathbb{B}\|_{2\alpha,[s,t]}|t-s|^{2\alpha}
  \\
    &+\left(\|B\|_{\alpha,[s,t]}\|R^{\sigma(X^\varepsilon)}\|_{2\alpha,[s,t]}
      +\|\mathbb{B}\|_{2\alpha,[s,t]}\|X^\varepsilon\|_{\alpha,[s,t]}\right)|t-s|^{3\alpha},
  \end{aligned}
\end{equation*}
where $\|X^\varepsilon\|_{\infty,[0,T]}:=\sup_{t\in[0,T]}|X_t^\varepsilon|$. Thus we can see that there exists a constant $M_1>0$ such that
\begin{equation}
 \label{Rh}
 \begin{aligned}
   \|R^{X^\varepsilon}\|_{2\alpha,[s,t]}&\le M_1\Bigl(
   \left(1+\|X^\varepsilon\|_{\infty,[0,T]}\right)|t-s|^{1-2\alpha}+\|\mathbb{B}\|_{2\alpha,[0,T]}\\ &+\left(\|B\|_{\alpha,[0,T]}\|R^{\sigma(X^\varepsilon)}\|_{2\alpha,[s,t]}
   +\|\mathbb{B}\|_{2\alpha,[0,T]}\|X^\varepsilon\|_{\alpha,[s,t]}\right)|t-s|^{\alpha}\Bigl).
 \end{aligned}
\end{equation}
Similar to the derivation of \eqref{Rh} and inequality $\left|\int_s^tb(X^\varepsilon_u,\theta_0)du\right|\lesssim\|X^\varepsilon\|_{\alpha,[s,t]}|t-s|^{1+\alpha}+(1+|X_s^\varepsilon)|)|t-s|$, there exists a constant $M_2>0$ such that
\begin{equation*}
  \begin{aligned}
    \|X^\varepsilon\|_{\alpha,[s,t]}&\le M_2\biggl(\|X^\varepsilon\|_{\alpha,[s,t]}|t-s|+(1+|X_s^\varepsilon)|)|t-s|^{1-\alpha}+\|\mathbb{B}\|_{2\alpha,[0,T]}|t-s|^\alpha
    \\
    &\hspace{50pt}+\left(\|B\|_{\alpha,[0,T]}\|R^{\sigma(X^\varepsilon)}\|_{2\alpha,[s,t]}
    +\|\mathbb{B}\|_{2\alpha,[0,T]}\|X^\varepsilon\|_{\alpha,[s,t]}\right)|t-s|^{2\alpha}+\|B\|_{\alpha,[0,T]}\biggl).
  \end{aligned}
\end{equation*}
Set
\begin{equation}
  \label{h}
  \delta=\left(2(1+M_1)(1+M_2)(1+2M_2e^{2M_2T}T)(1+\|\nabla_x\sigma\|_\infty)\left(1+\|B\|_{\alpha,[0,T]}+\|\mathbb{B}\|_{2\alpha,[0,T]}^{1/2}\right)\right)^{-1/\alpha},
\end{equation}
where $\|\nabla_x\sigma\|_\infty=\sup_{x\in\mathbb{R}^d}|\nabla_x\sigma(x)|$. Then for $|t-s|<\delta$,
\begin{equation}
  \label{Xh}
  \begin{aligned}
    \|X^\varepsilon\|_{\alpha,[s,t]}&\le 2M_2\left((1+|X_s^\varepsilon|)\delta^{1-\alpha}+\|\mathbb{B}\|_{2\alpha,[0,T]}\delta^\alpha
  +\|B\|_{\alpha,[0,T]}\|R^{\sigma(X^\varepsilon)}\|_{2\alpha,[s,t]}
    \delta^{2\alpha}+\|B\|_{\alpha,[0,T]}\right).
  \end{aligned}
\end{equation}
We aim to estimate the term $\|X^\varepsilon\|_{\infty,[0,T]}$.  By the triangle inequality and \eqref{Xh}
  \begin{equation*}
    \begin{aligned}
      |X^\varepsilon_t|&\le |X^\varepsilon_s|+\|X^\varepsilon\|_{\alpha,[s,t]}|t-s|^{\alpha}
      \\
      &\le|X_s^\varepsilon|(1+2M_2\delta)+2M_2\left(\delta+\|\mathbb{B}\|_{2\alpha,[0,T]}\delta^{2\alpha}
    +\|B\|_{\alpha,[0,T]}\|R^{\sigma(X^\varepsilon)}\|_{2\alpha,[s,t]}
      \delta^{3\alpha}+\|B\|_{\alpha,[0,T]}\delta^\alpha\right).
    \end{aligned}
  \end{equation*}
  for $|t-s|<\delta$. Iterating the above estimate for $N=T/\delta$, we can get
  \begin{equation}
    \label{5}
    \begin{aligned}
      &\|X^\varepsilon\|_{\infty,[0,T]}\\
      &\le|X_0|(1+2M_2\delta)^N
      \\
      &\hspace{20pt}+2M_2\left(\delta+\|\mathbb{B}\|_{2\alpha,[0,T]}\delta^{2\alpha}
    +\|B\|_{\alpha,[0,T]}\|R^{\sigma(X^\varepsilon)}\|_{2\alpha,[s,t]}
      \delta^{3\alpha}+\|B\|_{\alpha,[0,T]}\delta^\alpha\right)\sum_{k=1}^N(1+2M_2\delta)^k
      \\
      &\le|X_0|(1+2M_2\delta)^N
      \\
      &\hspace{20pt}+2M_2\left(\delta+\|\mathbb{B}\|_{2\alpha,[0,T]}\delta^{2\alpha}
    +\|B\|_{\alpha,[0,T]}\|R^{\sigma(X^\varepsilon)}\|_{2\alpha,[s,t]}
      \delta^{3\alpha}+\|B\|_{\alpha,[0,T]}\delta^\alpha\right)(1+2M_2\delta)^NN
      \\
      &\le |X_0|e^{2M_2T}+2M_2e^{2M_2T}\frac{T}{\delta}\left(\delta+\|\mathbb{B}\|_{2\alpha,[0,T]}\delta^{2\alpha}
    +\|B\|_{\alpha,[0,T]}\|R^{\sigma(X^\varepsilon)}\|_{2\alpha,[s,t]}
      \delta^{3\alpha}+\|B\|_{\alpha,[0,T]}\delta^\alpha\right),
    \end{aligned}
  \end{equation}
where we used the fact $(1+\frac{x}{N})^N\le e^x$ for $N\ge 0$. We estimate the term $\|R^{\sigma(X^\varepsilon)}\|_{2\alpha,[s,t]}$. By the definition of $R^{\sigma(X^\varepsilon)}$ and the Taylor expansion, we have
\begin{equation*}
  \begin{aligned}
    R^{\sigma(X^\varepsilon)}_{s,t}&=\sigma(X^\varepsilon_t)-\sigma(X^\varepsilon_s)-\varepsilon\nabla_x\sigma(X^\varepsilon_s)\sigma(X^\varepsilon_s)B_{s,t}\\
    &=\nabla_x\sigma(X_s^\varepsilon)X_{s,t}^\varepsilon+\int_0^1(1-\eta)\nabla_x^2\sigma(X^\varepsilon_s+\eta X^\varepsilon_{s,t})\langle X^\varepsilon_{s,t},X^\varepsilon_{s,t}\rangle d\eta-\varepsilon\nabla_x\sigma(X^\varepsilon_s)\sigma(X^\varepsilon_s)B_{s,t}
    \\
    &=\nabla_x\sigma(X_s^\varepsilon)R^{X^\varepsilon}_{s,t}+\int_0^1(1-\eta)\nabla_x^2\sigma(X^\varepsilon_s+\eta X^\varepsilon_{s,t})\langle X^\varepsilon_{s,t},X^\varepsilon_{s,t}\rangle d\eta,
  \end{aligned}
\end{equation*}
where $\nabla_x^2\sigma(x)\langle y,z\rangle=\sum_{i,j=1,\dots,d}\partial_{x_i}\partial_{x_j}\sigma(x)y_iz_j$. By Assumption \ref{(A2)}, it follows that
\begin{equation}
    \label{Rfh}
  \begin{aligned}
      \|R^{\sigma(X^\varepsilon)}\|_{2\alpha,\delta}
      \le \frac{\|\nabla_x^2\sigma\|_\infty}{2}\|X^\varepsilon\|_{\alpha,\delta}^2+\|\nabla_x\sigma\|\|R^{X^\varepsilon}\|_{2\alpha,\delta}.
  \end{aligned}
\end{equation}
Combining to \eqref{Rh}, \eqref{5} and \eqref{Rfh}, we have
\begin{equation*}
  \begin{aligned}
    &\|R^{X^\varepsilon}\|_{2\alpha,\delta}\le  M_1\biggl\{\left(\|B\|_{\alpha,[0,T]}  \|R^{\sigma(X^\varepsilon)}\|_{2\alpha,\delta}+\|\mathbb{B}\|_{2\alpha,[0,T]}
    \|X^\varepsilon\|_{\alpha,\delta}\right)\delta^{\alpha}\\
    &\hspace{200pt}+\|\mathbb{B}\|_{2\alpha,[0,T]}+\left(1+\|X^\varepsilon\|_{\infty,[0,T]}\right)\delta^{1-2\alpha}\biggl\}
    \\
    &\le M_1\biggl\{ \left(\|B\|_{\alpha,[0,T]}\|R^{\sigma(X^\varepsilon)}\|_{2\alpha,\delta}+\|\mathbb{B}\|_{2\alpha,[0,T]}
    \|X^\varepsilon\|_{\alpha,\delta}\right)\delta^{\alpha}+\|\mathbb{B}\|_{2\alpha,[0,T]}+\delta^{1-2\alpha}(1+|X_0|e^{2M_2T})
    \\
    &\hspace{50pt}+2M_2e^{2M_2T}T\left(\delta^{1-2\alpha}+\|\mathbb{B}\|_{2\alpha,[0,T]}
  +\|B\|_{\alpha,[0,T]}\|R^{\sigma(X^\varepsilon)}\|_{2\alpha,[s,t]}
    \delta^{\alpha}+\|B\|_{\alpha,[0,T]}\delta^{-\alpha}\right)\biggl\}
    \\
    &\le M_1\Biggl\{\left(\|B\|_{\alpha,[0,T]}\left(\frac{\|\nabla_x^2\sigma\|_\infty}{2}\|X^\varepsilon\|_{\alpha,\delta}^2+\|\nabla_x\sigma\|_\infty\|R^{X^\varepsilon}\|_{2\alpha,\delta}\right)(1+2M_2e^{2M_2T}T)\right)\delta^\alpha
    \\
    &\hspace{150pt}+\|\mathbb{B}\|_{2\alpha,[0,T]}
    \|X^\varepsilon\|_{\alpha,\delta}\delta^\alpha+\|\mathbb{B}\|_{2\alpha,[0,T]}(1+2M_2e^{2M_2T}T)
\\
&\hspace{150pt}+\left(1+|X_0|e^{2M_2T}+2M_2e^{2M_2T}T\right)+2M_2e^{2M_2T}T
    \|B\|_{\alpha,[0,T]}\delta^{-\alpha}\Biggl\},
  \end{aligned}
\end{equation*}
and by \eqref{h}
\begin{equation}
  \label{RX}
  \begin{aligned}
    &\|R^{X^\varepsilon}\|_{2\alpha,\delta}\le 2M_1\biggl\{\left(\|B\|_{\alpha,[0,T]}\frac{\|\nabla_x^2\sigma\|_\infty}{2}\|X^\varepsilon\|_{\alpha,\delta}^2(1+2M_2e^{2M_2T}T)+\|\mathbb{B}\|_{2\alpha,[0,T]}
    \|X^\varepsilon\|_{\alpha,\delta}\right)\delta^\alpha
    \\
    &+\|\mathbb{B}\|_{2\alpha,[0,T]}(1+2M_2e^{2M_2T}T)+\left(1+|X_0|e^{2M_2T}+2M_2e^{2M_2T}T\right)
    +2M_2e^{2M_2T}T
    \|B\|_{\alpha,[0,T]}\delta^{-\alpha}\biggl\}.
  \end{aligned}
\end{equation}
From the relation $X_{s,t}^\varepsilon=\sigma(X_s^\varepsilon)B_{s,t}+R^{X^\varepsilon}_{s,t}$, we have
\begin{equation}
  \label{4}
  \begin{aligned}
    &\|X^\varepsilon\|_{\alpha,[s,t]}\le\|B\|_{\alpha,[0,T]}+\|R^{X^\varepsilon}\|_{2\alpha,\delta}\delta^{\alpha}
    \\    &\le \|B\|_{\alpha,[0,T]}(1+4M_1M_2e^{2M_2T}T)\\
    &\hspace{20pt}+2M_1\biggl\{\left(\|B\|_{\alpha,[0,T]}\frac{\|\nabla_x^2\sigma\|_\infty}{2}\|X^\varepsilon\|_{\alpha,\delta}^2(1+2M_2e^{2M_2T}T)+\|\mathbb{B}\|_{2\alpha,[0,T]}
    \|X^\varepsilon\|_{\alpha,\delta}\right)\delta^{2\alpha}
    \\
    &\hspace{20pt}+\|\mathbb{B}\|_{2\alpha,[0,T]}\delta^\alpha(1+2M_2e^{2M_2T}T)+\left(1+|X_0|e^{2M_2T}+2M_2e^{2M_2T}T\right)\biggl\}.
  \end{aligned}
\end{equation}
By \eqref{h}, there exists a constant $C>0$ such that
\begin{equation*}
  \begin{aligned}
      \|X^\varepsilon\|_{\alpha,\delta}\le  C\left(1+\|X^\varepsilon\|_{\alpha,\delta}^2\delta^\alpha+\|B\|_{\alpha,[0,T]}+\|\mathbb{B}\|_{2\alpha,[0,T]}^{1/2}\right).
  \end{aligned}
\end{equation*}
Multiplying both sides by $C\delta^\alpha$, we obtain that
\begin{equation}
  \label{6}
  \|X^\varepsilon\|_{\alpha,\delta}\lesssim1+\|(B,\mathbb{B})\|_{\alpha,[0,T]}.
\end{equation}
Applying Lemma \ref{HHE}, we conclude that
\begin{equation*}
  \|X^\varepsilon\|_{\alpha,[0,T]}\lesssim\left(1+\|(B,\mathbb{B})\|_{\alpha,[0,T]}\right)\vee\left(1+\|(B,\mathbb{B})\|_{\alpha,[0,T]}^{1/\alpha}\right)
  \lesssim1+\|(B,\mathbb{B})\|_{\alpha,[0,T]}^{1/\alpha}.
\end{equation*}
By combining \eqref{RX}, \eqref{6}, Lemma \ref{HHE} and \eqref{Rfh}, \eqref{6}, Lemma \ref{HHE}, we can derive \eqref{RLX} and \eqref{RLSX} respectively and we complete the proof of Lemma \ref{HE}.
 \end{proof}
 \section{Construction of the maximum likelihood type estimator}
 In this section, we construct the maximum likelihood type estimator from observed data $\{X_t^\varepsilon\}_{t\in[0,T]}$. We impose some further assumptions to the function $\sigma$. Let $A(x):=[\sigma\sigma^*](x),~A^\prime(x)=[\sigma^*\sigma](x)$.
  \begin{assumption}
    \label{A3}
   \begin{equation*}
     \inf_{x\in\mathbb{R}^d}\mathrm{det}A(x)>0, ~\inf_{x\in\mathbb{R}^d}\mathrm{det}A^\prime(x)>0.
   \end{equation*}
 \end{assumption}
 \begin{assumption}
   \label{sym}
   For every $1\le i\le d,~1\le j\le r$ and $x\in\mathbb{R}^d$,
   \begin{equation*}
      \sum_{k=1}^d\frac{\partial\sigma^{(i,j)}}{\partial x_k}(x)\sigma^{(k,l)}(x)=\sum_{k=1}^d\frac{\partial\sigma^{(i,l)}}{\partial x_k}(x)\sigma^{(k,j)}(x).
   \end{equation*}
 \end{assumption}
 We first recall the basic definitions of fractional calculus. Let $f\in L^1(a,b)$ for $a<b$ and $\alpha>0$. The fractional Riemann--Liouville integrals of $f$ of order $\alpha$ are defined for almost all $x\in(a,b)$ by
\begin{equation*}
 I_{a+}^{\alpha}f(x):=\frac{1}{\Gamma(\alpha)}\int_a^x(x-y)^{\alpha-1}f(y)dy,
\end{equation*}
and
\begin{equation*}
 I_{b-}^{\alpha}f(x):=\frac{1}{\Gamma(\alpha)}\int_x^b(y-x)^{\alpha-1}f(y)dy.
\end{equation*}
 For $1\le i\le r$, define the function $Q_{H,\theta}^\varepsilon(t)=\left(  Q_{H_1,\theta}^\varepsilon(t),\dots,Q_{H_r,\theta}^\varepsilon(t)\right)$ by
\begin{equation*}
  Q_{H_i,\theta}^\varepsilon(t):=
    \left(\varepsilon d_{H_i}\right)^{-1}t^{H_i-1/2}I_{0+}^{1/2-H_i}\left[(\cdot)^{1/2-H_i}\left(\sigma^*(X_\cdot^\varepsilon)A^{-1}(X_\cdot^\varepsilon)b(X^\varepsilon_\cdot,\theta)\right)^i\right](t),
\end{equation*}
where
\begin{equation*}
  d_{H_i}:=\sqrt{\frac{2H_i\Gamma(\frac{3}{2}-H_i)\Gamma(H_i+\frac{1}{2})}{\Gamma(2-2H_i)}}.
\end{equation*}
 For $0<s<t$, let $k_{H_i}(t,s)$ and $k_{H_i}^{-1}(t,s)$ be functins given by
 \begin{equation*}
   \begin{aligned}
       k_{H_i}(t,s):=&d_{H_i}s^{1/2-H_i}D_{t-}^{1/2-H_i}\left[(\cdot)^{H_i-1/2}\right]\\
       k_{H_i}^{-1}(t,s):=&\frac{1}{d_{H_i}}s^{1/2-H_i}I_{t-}^{1/2-H_i}\left[(\cdot)^{H_i-1/2}\right].
   \end{aligned}
 \end{equation*}
We transform the solution of \eqref{SDE} to a fractional Brownian motion with drift.
\begin{lemma}
  \label{transform Y}
  Let $\mathcal{P}$ be the partition of the segment $[0,t]$. Then under Assumptions \ref{(A1)}-\ref{sym},
  \begin{equation*}
    \begin{aligned}
        Y_t:=&\lim_{|\mathcal{P}|\rightarrow 0}\sum_{t_k,t_{k+1}\in\mathcal{P}}\varepsilon^{-1}\biggl(\sigma^*(X_{t_k}^\varepsilon)A^{-1}(X_{t_k}^\varepsilon)X_{t_k,t_{k+1}}^\varepsilon+\frac{1}{2}\nabla_x\left(\sigma^*(X_{t_k}^\varepsilon)A^{-1}
        (X_{t_k}^\varepsilon)\right)(X^\varepsilon_{t_k,t_{k+1}})^{\otimes2}\biggl)
      \\
      =&\varepsilon^{-1}\int_0^t\sigma^*(X_s^\varepsilon)A^{-1}(X_s^\varepsilon)b(X_s^\varepsilon,\theta_0)ds+B_t.
    \end{aligned}
  \end{equation*}
\end{lemma}
\begin{proof}
 Elementary computations yield
  \begin{equation*}
    \begin{aligned}
    Y_t=&\lim_{|\mathcal{P}|\rightarrow 0}\sum_{t_k,t_{k+1}\in\mathcal{P}}\varepsilon^{-1}\biggl(\sigma^*(X_{t_k}^\varepsilon)A^{-1}(X_{t_k}^\varepsilon)X_{t_k,t_{k+1}}^\varepsilon+\frac{1}{2}\nabla_x\left(\sigma^*(X_{t_k}^\varepsilon)A^{-1}
    (X_{t_k}^\varepsilon)\right)(X^\varepsilon_{t_k,t_{k+1}})^{\otimes2}\biggl)
    \\
    &=\lim_{|\mathcal{P}|\rightarrow 0}\sum_{t_k,t_{k+1}\in\mathcal{P}}\varepsilon^{-1}\biggl(\sigma^*(X_{t_k}^\varepsilon)A^{-1}(X_{t_k}^\varepsilon)\left\{\int_{t_k}^{t_{k+1}}b(X_s^\varepsilon,\theta_0)ds
    +\varepsilon\int_{t_k}^{t_{k+1}}\sigma(X_s^\varepsilon)dB_s\right\}
    \\
    &\hspace{200pt}+\frac{1}{2}\nabla_x\left(\sigma^*(X_{t_k}^\varepsilon)A^{-1}
    (X_{t_k}^\varepsilon)\right)(X_{t_k,t_{k+1}})^{\otimes2}\biggl)
    \\
    &=\varepsilon^{-1}\int_0^t\sigma^*(X_s^\varepsilon)A^{-1}(X_s^\varepsilon)b(X_s^\varepsilon,\theta_0)ds+B_t
    \\
    &+\varepsilon^{-1}\lim_{|\mathcal{P}|\rightarrow 0}\sum_{t_k,t_{k+1}\in\mathcal{P}}\biggl(\varepsilon\sigma^*(X_{t_k}^\varepsilon)
    A^{-1}(X_{t_k}^\varepsilon)\int_{t_k}^{t_{k+1}}\sigma(X_s^\varepsilon)dB_s-\varepsilon\sigma^*(X_{t_k}^\varepsilon)
    A^{-1}(X_{t_k}^\varepsilon)\sigma(X_{t_k}^\varepsilon)B_{t_k,t_{k+1}}
\\
    &\hspace{200pt}+\frac{1}{2}\nabla_x\left(\sigma^*(X_{t_k}^\varepsilon)A^{-1}(X_{t_k}^\varepsilon)\right)(X^\varepsilon_{t_k,t_{k+1}})^{\otimes2}\biggl).
    \end{aligned}
  \end{equation*}
  We show that the limit of the third term tends to zero. Note that
  \begin{equation*}
    \begin{aligned}
      \nabla_x\left(\sigma^*(x)A^{-1}(x)\right)&=\nabla_x\left(\sigma^*(x)\right)A^{-1}(x)+\sigma^*(x)\nabla_xA^{-1}(x)
      \\
      &=\nabla_x\left(\sigma^*(x)\right)A^{-1}(x)-\sigma^*(x)A^{-1}(x)\left(\nabla_x\left(\sigma(x)\right)\sigma^*(x)+\sigma(x)\nabla_x\sigma^*(x)\right)A^{-1}(x)
      \\
      &=-\sigma^*(x)A^{-1}(x)\nabla_x\left(\sigma(x)\right)\sigma^*(x)A^{-1}(x),~~~x\in\mathbb{R}^d.
    \end{aligned}
  \end{equation*}
  From Proposition \ref{rie}, it follows that
  \begin{equation}
    \label{te}
    \begin{aligned}
      X^\varepsilon_{t_k,t_{k+1}}&=\int_{t_k}^{t_{k+1}}b(X_s^\varepsilon,\theta_0)ds+\varepsilon\int_{t_k}^{t_{k+1}}\sigma(X_s^\varepsilon)dB_s
      \\
      &=\varepsilon\sigma(X_{t_k}^\varepsilon)B_{t_k,t_{k+1}}+\varepsilon^2\nabla_x\sigma(X_{t_k}^\varepsilon)\sigma(X_{t_k}^\varepsilon)\mathbb{B}_{t_k,t_{k+1}}+O(|t_{k+1}-t_k|).
    \end{aligned}
  \end{equation}
  Calculating the tendor product $(X^\varepsilon_{t_k,t_{k+1}})^{\otimes2}$ with the relation \eqref{te}, we have
  \begin{equation*}
    \begin{aligned}
      (X^\varepsilon_{t_k,t_{k+1}})^{\otimes2}=&\left(\varepsilon\sigma(X_{t_k}^\varepsilon)B_{t_k,t_{k+1}}\right)^{\otimes 2}+\left(\varepsilon^2\nabla_x\sigma(X_{t_k}^\varepsilon)\sigma(X_{t_k}^\varepsilon)\mathbb{B}_{t_k,t_{k+1}}\right)^{\otimes 2}
      \\
      &+\varepsilon^3\sigma(X_{t_k}^\varepsilon)B_{t_k,t_{k+1}}\otimes\nabla_x\sigma(X_{t_k}^\varepsilon)\sigma(X_{t_k}^\varepsilon)\mathbb{B}_{t_k,t_{k+1}}
      \\
      &+\varepsilon^3\nabla_x\sigma(X_{t_k}^\varepsilon)\sigma(X_{t_k}^\varepsilon)\mathbb{B}_{t_k,t_{k+1}}\otimes\sigma(X_{t_k}^\varepsilon)B_{t_k,t_{k+1}}+O\left(|t_{k+1}-t_k|^{1+H}\right)
      \\
      &=\left(\varepsilon\sigma(X_{t_k}^\varepsilon)B_{t_k,t_{k+1}}\right)^{\otimes 2}+O\left(|t_{k+1}-t_k|^{3H}\right).
    \end{aligned}
  \end{equation*}
  Therefore,
  \begin{equation*}
    \begin{aligned}
      &\varepsilon^{-1}\biggl(\varepsilon\sigma^*(X_{t_k}^\varepsilon)
      A^{-1}(X_{t_k}^\varepsilon)\int_{t_k}^{t_{k+1}}\sigma(X_s^\varepsilon)dB_s-\varepsilon\sigma^*(X_{t_k}^\varepsilon)
      A^{-1}(X_{t_k}^\varepsilon)\sigma(X_{t_k}^\varepsilon)B_{t_k,t_{k+1}}
  \\
      &\hspace{200pt}+\frac{1}{2}\nabla_x\left(\sigma^*(X_{t_k}^\varepsilon)A^{-1}(X_{t_k}^\varepsilon)\right)(X^\varepsilon_{t_k,t_{k+1}})^{\otimes2}\biggl)
      \\
      &=\sigma^*(X_{t_k}^\varepsilon)
      A^{-1}(X_{t_k}^\varepsilon)\biggl(\int_{t_k}^{t_{k+1}}\sigma(X_s^\varepsilon)dB_s-\sigma(X_{t_k}^\varepsilon)B_{t_k,t_{k+1}}
      \\
      &\hspace{100pt}-\frac{\varepsilon}{2}\nabla_x\left(\sigma(X_{t_k}^\varepsilon)\right)\sigma^*(X_{t_k}^\varepsilon)
          A^{-1}(X_{t_k}^\varepsilon)\left(\sigma\left(X_{t_k}^\varepsilon\right)B_{t_k,t_{k+1}}\right)^{\otimes2}\biggl)+O\left(|t_{k+1}-t_k|^{3H}\right)
            \\
      &=\sigma^*(X_{t_k}^\varepsilon)
      A^{-1}(X_{t_k}^\varepsilon)\left(\int_{t_k}^{t_{k+1}}\sigma(X_s^\varepsilon)dB_s-\sigma(X_{t_k}^\varepsilon)B_{t_k,t_{k+1}}-\varepsilon\nabla_x\left(\sigma(X_{t_k}^\varepsilon)\right)\sigma(X_{t_k}^\varepsilon)\mathbb{B}_{t_k,t_{k+1}}\right)
      \\
      &\hspace{50pt}-\varepsilon\sigma^*(X_{t_k}^\varepsilon)
        A^{-1}(X_{t_k}^\varepsilon)\nabla_x\left(\sigma(X_{t_k}^\varepsilon)\right)\sigma(X_{t_k}^\varepsilon)\left(\frac{1}{2}(B_{t_k,t_{k+1}})^{\otimes 2}-\mathbb{B}_{t_k,t_{k+1}}\right)+O\left(|t_{k+1}-t_k|^{3H}\right).
    \end{aligned}
  \end{equation*}
  Since $(B,\mathbb{B})$ is a geometirc rough path, we have ${\rm Sym}(\mathbb{B}_{t_k,t_{k+1}})=\frac{1}{2}(B_{t_k,t_{k+1}})$ and $\frac{1}{2}(B_{t_k,t_{k+1}})^{\otimes 2}-\mathbb{B}_{t_k,t_{k+1}}={\rm Anti}(\mathbb{B}_{t_k,t_{k+1}})$
  where ${\rm Anti}(\mathbb{B}_{t_k,t_{k+1}})=\frac{1}{2}(\mathbb{B}-\mathbb{B}^*)$ is the antisymmetric part of $\mathbb{B}$. Note that for $1\le i\le d$
  \begin{equation*}
    \begin{aligned}
      \left(\nabla_x\left(\sigma(X_{t_k}^\varepsilon)\right)\sigma(X_{t_k}^\varepsilon){\rm Anti}(\mathbb{B}_{t_k,t_{k+1}})\right)^{(i)}&=\sum_{j,k,l,m}\frac{\partial \sigma^{(i,l)}(X_{t_k}^\varepsilon)}{\partial x_j}
      \sigma^{(j,m)}(X_{t_k}^\varepsilon){\rm Anti}(\mathbb{B}_{t_k,t_{k+1}})^{(l,m)}
      \\
      &=-\sum_{j,k,l,m}\frac{\partial \sigma^{(i,l)}(X_{t_k}^\varepsilon)}{\partial x_j}
      \sigma^{(j,m)}(X_{t_k}^\varepsilon){\rm Anti}(\mathbb{B}_{t_k,t_{k+1}})^{(m,l)}.
    \end{aligned}
  \end{equation*}
  By using Assumption \ref{sym}, we can deduce $  \nabla_x\left(\sigma(X_{t_k}^\varepsilon)\right)\sigma(X_{t_k}^\varepsilon){\rm Anti}(\mathbb{B}_{t_k,t_{k+1}})=0$ and
  \begin{equation*}
    \begin{aligned}
      \varepsilon^{-1}\sum_{t_k,t_{k+1}\in\mathcal{P}}\biggl(\varepsilon\sigma^*(X_{t_k}^\varepsilon)
      A^{-1}(X_{t_k}^\varepsilon)&\int_{t_k}^{t_{k+1}}\sigma(X_s^\varepsilon)dB_s-\varepsilon\sigma^*(X_{t_k}^\varepsilon)
      A^{-1}(X_{t_k}^\varepsilon)\sigma(X_{t_k}^\varepsilon)B_{t_k,t_{k+1}}
      \\
      &+\frac{1}{2}\nabla_x\left(\sigma^*(X_{t_k}^\varepsilon)A^{-1}(X_{t_k}^\varepsilon)\right)(X^\varepsilon_{t_k,t_{k+1}})^{\otimes2}\biggl)
      =O\left(|\mathcal{P}|^{3H-1}\right).
    \end{aligned}
  \end{equation*}
  Passing to the limit as $|\mathcal{P}|\rightarrow 0$, we complete the proof of Lemma \ref{transform Y}.
\end{proof}


Define a stochastic process $W_t=(W_t^1,\dots,W_t^r)$ by
\begin{equation*}
  W_t^{i}:=\int_0^tk_{H_i}^{-1}(t,s)dB_s^i,~~~1\le i\le r.
\end{equation*}
 Here we interpret the stochastic integral with respect to a fractional Brownian motion as a Wiener integral and then we can find that $\{W_t\}_{0\le t\le T}$ is r--dimensional Wiener process.
We define a semimartingale $Z_t=\left(Z_t^1,\dots,Z_t^r\right)$ as follows:
\begin{equation*}
 \begin{aligned}
   Z_t^i:&=\int_0^Tk_{H_i}^{-1}(t,s)dY_s^i
   =\int_0^tQ_{H_i,\theta_0}^\varepsilon(s)ds+W_t^i,~~~1\le i\le r.
 \end{aligned}
\end{equation*}
 Let $\tilde{P}$ be a probability measure defined by $d\tilde{P}/dP=\xi_T^{\theta_0}$ where
\begin{equation*}
  \xi_T^{\theta_0}=\exp\left(-\int_0^T\left(Q_{H,\theta_0}^\varepsilon(t)\right)^*dW_t-\frac{1}{2}\int_0^T\left(Q_{H,\theta_0}^\varepsilon(t)\right)^*Q_{H,\theta_0}^\varepsilon(t)dt\right).
\end{equation*}
Now, let us assume that
\begin{align}
E\xi_T^{\theta_0}=1, \label{density}
\end{align}
one of the sufficient conditions for which is given in Lemma \ref{absolute continuity} later.
By the standard Girsanov theorem, we obtain that the stochastic process $\{Z_t\}_{t\in[0,T]}$ is a r--dimensional $\mathcal{F}_t^{B}$-Brownian motion under $\tilde{P}$. Hence we see that
\begin{equation*}
\int_0^tk_{H_i}(t,s)dZ_t^i=Y_t^i.
\end{equation*}
 is a $(\mathcal{F}_t^{B})$ fractional Brownian motion under $\tilde{P}$.

Define an enhanced process of $Y$ as follows:
\begin{equation*}
  \begin{aligned}
    \mathbb{Y}_{s,t}&=\left(\mathbb{Y}_{s,t}^{i,j}\right)_{i,j=1}^r,
    \\
    \mathbb{Y}_{s,t}^{i,j}&\overset{L^2}{=}\lim_{|\mathcal{P}|\rightarrow 0}
    \sum_{t_k,t_{k+1}\in\mathcal{P}}Y_{s,t_k}^{i}Y_{t_k,t_{k+1}}^{j}
    \\
    &=\mathbb{B}^{i,j}_{s,t}+\varepsilon^{-2}\int_s^t\left(\int_s^v\left[\sigma^*(X_u^\varepsilon)A^{-1}(X_u^\varepsilon)b(X_u^\varepsilon,\theta_0)\right]^idu\right)\left[\sigma^*(X_v^\varepsilon)A^{-1}(X_v^\varepsilon)b(X_v^\varepsilon,\theta_0)\right]^{j}dv
    \\
    &~~~+\varepsilon^{-1}\int_s^tB_{s,v}^i\left[\sigma^*(X_v^\varepsilon)A^{-1}(X_v^\varepsilon)b(X_v^\varepsilon,\theta_0)\right]^{j}dv
    \\
    &~~~+\varepsilon^{-1}\int_s^t\left(\int_s^v\left[\sigma^*(X_u^\varepsilon)A^{-1}(X_u^\varepsilon)b(X_u^\varepsilon,\theta_0)\right]^idu\right)dB^{j}_v,~~~1\le i<j\le r,
    \\
    \mathbb{Y}_{s,t}^{i,i}&=\frac{1}{2}\left(Y^i_{s,t}\right)^2,~~~1\le i\le r,\\
    \mathbb{Y}_{s,t}^{i,j}&=-\mathbb{Y}^{j,i}_{s,t}+Y_{s,t}^iY_{s,t}^{j},~~~
    1\le j<i\le r.
  \end{aligned}
\end{equation*}
for $s,t\in[0,T]$, where $\mathcal{P}=\left\{s=t_0<t_1<\cdots<t_n=t\right\}$ is an arbitrary partition of the segment $[s,t]$. We remark that each integral appearing in the process $\mathbb{Y}$ are defined by the Young pathwise integral.
 By Proposition \ref{rfb} and algebraic identities for Riemann sums, we can show that  $(Y,\mathbb{Y})\in\mathscr{C}^\alpha([0,T],\mathbb{R}^r)$. In addition,
\begin{equation*}
  \begin{aligned}
    &\sigma(X_t^\varepsilon)-\sigma(X_s^\varepsilon)=\nabla_x\sigma(X_s^\varepsilon)X_{s,t}^\varepsilon+\int_0^1(1-\tau)\nabla_x^2\sigma(X^\varepsilon_s+\tau X_{s,t}^\varepsilon)\langle X^\varepsilon_{s,t},X^\varepsilon_{s,t}\rangle d\tau
    \\
    &=\varepsilon\nabla_x\sigma(X_s^\varepsilon)\sigma(X_s^\varepsilon)B_{s,t}+ \nabla_x\sigma(X_s^\varepsilon)\sigma(X_s^\varepsilon)\int_s^t\sigma^*(X_u^\varepsilon)A^{-1}(X_u^\varepsilon)b(X_u^\varepsilon,\theta_0)du+Y^\sharp_{s,t}
    \\
    &= \varepsilon\nabla_x\sigma(X_s^\varepsilon)\sigma(X_s^\varepsilon)Y_{s,t}+Y^\sharp_{s,t},
  \end{aligned}
\end{equation*}
where
\begin{equation*}
  \begin{aligned}
      Y^\sharp_{s,t}=\int_0^1(1-\tau)\nabla_x^2\sigma(X^\varepsilon_s+\tau X_{s,t}^\varepsilon)\langle X^\varepsilon_{s,t},X^\varepsilon_{s,t}\rangle d\tau+\nabla_x\sigma(X_s^\varepsilon)R^{X^\varepsilon}_{s,t}
      \\
      -\nabla_x\sigma(X_s^\varepsilon)\sigma(X_s^\varepsilon)\int_s^t\sigma^*(X_u^\varepsilon)A^{-1}(X_u^\varepsilon)b(X_u^\varepsilon,\theta_0)du.
  \end{aligned}
\end{equation*}
 Since
\begin{equation*}
\sup_{s,t\in[0,T]: s\ne t}\frac{|Y^\sharp_{s,t}|}{|t-s|^{2\alpha}}<\infty,
\end{equation*}
 we conclude that
  $\left(\sigma(X^\varepsilon),\varepsilon\nabla_x\sigma(X_s^\varepsilon)\sigma(X_s^\varepsilon)\right)\in\mathcal{D}^\alpha_{Y}([0,T],\mathbb{R}^r)$.
Therefore, we can define the rough integral $\int\sigma(X_s^\varepsilon)dY_s$ and have the following equality $\tilde{P}$-a.s.,
\begin{equation}
  \label{siy}
  \begin{aligned}
    &\varepsilon\int_0^t\sigma(X_s^\varepsilon)dY_s=\lim_{|\mathcal{P}|\rightarrow 0}\sum_{t_k,t_{k+1}\in\mathcal{P}}\varepsilon\left(\sigma(X_{t_k}^\varepsilon)Y_{t_k,t_{k+1}}
        +\varepsilon\nabla_x\sigma(X_s^\varepsilon)\sigma(X_s^\varepsilon)\mathbb{Y}_{t_k,t_{k+1}}\right)
        \\
        &=\lim_{|\mathcal{P}|\rightarrow 0}\sum_{t_k,t_{k+1}\in\mathcal{P}}\Biggl(\sigma(X_{t_k}^\varepsilon)\int_{t_k}^{t_{k+1}}\sigma^*(X_s^\varepsilon)A^{-1}(X_s^\varepsilon)b(X_s^\varepsilon,\theta_0)ds
        +\varepsilon\sigma(X_{t_k}^\varepsilon)B_{t_k,t_{k+1}}\\
        &\hspace{50pt}+\varepsilon^2 \nabla_x\sigma(X_{t_k}^\varepsilon)\sigma(X_{t_k}^\varepsilon)\mathbb{B}_{t_k,t_{k+1}}+\varepsilon^2 \nabla_x\sigma(X_s^\varepsilon)\sigma(X_s^\varepsilon)(\mathbb{Y}_{t_k,t_{k+1}}-\mathbb{B}_{t_k,t_{k+1}})\Biggl)
        \\
        &=\int_0^tb(X_s^\varepsilon,\theta_0)ds+\varepsilon\int_0^t\sigma(X^\varepsilon_s)dB_s.
  \end{aligned}
\end{equation}
Let $(X^{\varepsilon,0},\varepsilon\sigma(X^{\varepsilon,0}))\in\mathcal{D}_B^{2\alpha}([0,T],\mathbb{R}^d)$ be the solution to the following equation:
\begin{equation*}
 X_t^{\varepsilon,0}=X_0+\varepsilon\int_0^t\sigma(X_s^{\varepsilon,0})\,dB_s.
\end{equation*}
By \eqref{siy}, the process $\{X^{\varepsilon,0}\}_{t\in[0,T]}$ considered on the probability space $(\Omega,\mathcal{F},P)$, satisfies the same equation as the process $\{X_t^\varepsilon\}_{t\in[0,T]}$ on $(\Omega,\mathcal{F},\tilde{P})$.
 Let $P_{\theta_0}^\varepsilon$ and $P_0^\varepsilon$ be probability measures induced by processes $\{X^{\varepsilon}_t\}_{t\in[0,T]} $ and $\{X^{\varepsilon,0}_t\}_{t\in[0,T]}$, respectively.
Then for every $A\in\mathcal{B}(C[0,T])$,
\begin{equation*}
P_{0}^\varepsilon(A)=P(X^{\varepsilon,0}\in A)=\tilde{P}(X^{\varepsilon}\in A)=\int_{\left\{X^\varepsilon\in A\right\}}\xi_T^{\theta_0}(\omega)dP(\omega),
\end{equation*}
and by Lemma 6.8 in \cite{Liptser}, we find that the Radon-Nikodym derivative of $P^\varepsilon_{\theta_0}$ with respect to $P_0^\varepsilon$ is given by
\begin{equation}
  \label{RN}
  \frac{dP_{\theta_0}^\varepsilon}{dP_0^\varepsilon}(X^\varepsilon)=\exp\left(\int_0^T\left(Q_{H,\theta_0}^\varepsilon(t)\right)^*dW_t+\frac{1}{2}\int_0^T\left(Q_{H,\theta_0}^\varepsilon(t)
  \right)^*Q_{H,\theta_0}^\varepsilon(t)\,dt\right).
\end{equation}

Let us make remark for \eqref{RN}.
We have derived Radon--Nikodym derivative \eqref{RN} via the Girsanov theorem under the assumption \eqref{density}. However, to ensure the condition
\eqref{density}, we need a stronger assumption for coefficients $b$ and $\sigma$.
\begin{assumption}
  \label{ac}
  There exists a constant $C>0$ and $\lambda\in(0,h)$ such that for every $\theta\in\Theta$ and $y\in\mathbb{R}^d$,
  \begin{equation*}
    \left|\sigma^*(y)A^{-1}(y)b(y,\theta)\right|\le C(1+|y|^\lambda).
  \end{equation*}
\end{assumption}
\begin{lemma}
  \label{absolute continuity}
  Suppose that Assumptions \ref{(A1)}-\ref{ac} are satisfied. Then the stochastic process $\{\Lambda_t^{\theta_0}\}_{t\in[0,T]}$ is a martingale.
\end{lemma}
\begin{proof}
  It is enough to show that there exists a constant $M>0$ such that
  \begin{equation}
    \sup_{t\in[0,T]}E\exp\left(M\left(Q_{H,\theta}^\varepsilon(t)\right)^*Q_{H,\theta}^\varepsilon(t)\right)<\infty.
  \end{equation}
  By using Lemma \ref{HE} with $0<\lambda<\alpha<h$, there exist constants $M_1,M_2>0$ such that
  \begin{equation*}
    \begin{aligned}
          &\exp\left(M\left(Q_{H,\theta}^\varepsilon(t)\right)^*Q_{H,\theta}^\varepsilon(t)\right)
\\
          &=\exp\left(M\sum_{i=1}^r\left(\varepsilon d_{H_i}\right)^{-1}t^{2H_i-1}\left(\int_0^ts^{1/2-H_i}(t-s)^{-1/2-H_i}\left[\sigma^*(X_s^\varepsilon)A^{-1}(X_s^\varepsilon)b(X_s^\varepsilon,\theta_0)\right]_ids\right)^2\right)
          \\
          &\le\exp\left(M_1\left(1+\|X^\varepsilon\|_{\infty,[0,T]}^{2\lambda}\right)\sum_{i=1}^rt^{2H_i-1}\left(\int_0^ts^{1/2-H_i}(t-s)^{-1/2-H_i}ds\right)^2\right)\\
          &\le\exp\left(M_2\left(1+\|(B,\mathbb{B})\|_{\alpha,[0,T]}^{2\lambda/\alpha}\right)\sum_{i=1}^rt^{1-2H_i}\beta(3/2-H_i,1/2-H_i)^2\right),
    \end{aligned}
  \end{equation*}
  and by Fernique's theorem for Gaussian rough path (see Theorem 11.9 in \cite{Fritz}), we complete the proof of Lemma \ref{absolute continuity}.
\end{proof}
To summarize the discussion, the following theorem holds.
\begin{theorem}
  \label{ABC}
  Suppose that Assumptions \ref{(A1)}-\ref{ac} are satisfied. Then the probability measures $P_{\theta_0}^\varepsilon$ and $P_0^\varepsilon$ are equivalent and
  \begin{equation*}
    \frac{dP_{\theta_0}^\varepsilon}{dP_0^\varepsilon}(X^\varepsilon)=\exp\left(\int_0^T\left(Q_{H,\theta_0}^\varepsilon(t)\right)^*dZ_t-\frac{1}{2}\int_0^T\left(Q_{H,\theta_0}^\varepsilon(t)
    \right)^*Q_{H,\theta_0}^\varepsilon(t)\,dt\right).
  \end{equation*}
\end{theorem}

 We consider the log-likelihood type function $\mathbb{L}_{H,\varepsilon}$
\begin{equation*}
 \mathbb{L}_{H,\varepsilon}(\theta):=\int_0^T\left(Q_{H,\theta}^\varepsilon(t)\right)^*dZ_t-\frac{1}{2}\int_0^T\left(Q_{H,\theta}^\varepsilon(t)\right)^*Q_{H,\theta}^\varepsilon(t)dt,
\end{equation*}
 and define the maximum likelihood type estimator by
\begin{equation}
 \hat{\theta}_\varepsilon:=\argmax_{\theta\in\bar{\Theta}}\mathbb{L}_{H,\varepsilon}(\theta).\label{likelihood}
\end{equation}
\begin{remark}
  We remark that Theorem \ref{ABC} asserts that \eqref{likelihood} is truly MLE. However, if estimation is the only goal, the estimator need not be MLE. Therefore, Assumption \ref{ac}, which we introduced to guarantee \eqref{density}, is not imposed when we derive the asymptotic properties of \eqref{likelihood}.
\end{remark}
\begin{remark}
  \label{dif}
  The log-likelihood function $\mathbb{L}_{H,\varepsilon}$ is differentiable in $\theta$ under Assumption \ref{(A1)}-\ref{A3}, and we have
  \begin{equation*}
    \begin{aligned}
      \nabla_\theta\mathbb{L}_{H,\varepsilon}(\theta)&=\sum_{i=1}^r\left(\int_0^T\nabla_\theta Q_{H_i,\theta}^\varepsilon(t)dZ_t^i-
      \int_0^TQ_{H_i,\theta}^\varepsilon(t)\nabla_\theta Q_{H_i,\theta}^\varepsilon(t)dt\right),
      \\
      \nabla_\theta^2\mathbb{L}_{H,\varepsilon}(\theta)&=\sum_{i=1}^r\Biggl(\int_0^T\nabla_\theta^2 Q_{H_i,\theta}^\varepsilon(t)dZ_t^i-\int_0^T\left(\nabla_\theta Q_{H_i,\theta}^\varepsilon(t)\right)^{\otimes2}dt
      \\
      &\hspace{180pt}-\int_0^TQ_{H_i,\theta}^\varepsilon(t)\nabla_\theta^2 Q_{H_i,\theta}^\varepsilon(t)dt\Biggl).
    \end{aligned}
  \end{equation*}
\end{remark}
In order to state asymptotic properties of the estimator \eqref{likelihood}, we make some notations. Let $\{x_t\}_{0\le t\le T}$ be the solution to the differential equation under the true value of the drift parameter:
\begin{equation}
 \label{ODE}
  \left \{
  \begin{aligned}
    \frac{dx_t}{dt}&=b(x_t,\theta_0)\\
    x_0&=X_0.
  \end{aligned}\right .
\end{equation}
 Define the function $Q_{H,\theta}=\left(Q_{H_{1},\theta},\dots,Q_{H_{r},\theta}\right)$ by
 \begin{equation*}
   Q_{H_i,\theta}(t):=d_{H_i}^{-1}t^{H_i-1/2}I^{1/2-H_i}_{0+}\left[(\cdot)^{1/2-H_i}\left(\sigma^*(x_\cdot)A^{-1}(x_\cdot)b(x_\cdot,\theta)\right)_i\right](t)
 \end{equation*}
 and we set a $m\times m$ square matrix $\Gamma_H(\theta_0)$ which is an asymptotic convariance matrix of the estimator \eqref{likelihood} as
\begin{equation*}
  \begin{aligned}
    \Gamma_H^{j,k}(\theta_0):&=\int_0^T\left(\partial_{\theta_j}Q_{H,\theta_0}(t)\right)^*\partial_{\theta_k}Q_{H,\theta_0}(t)dt
    \\
    &=\sum_{i=1}^r\gamma_i\int_0^Tt^{2H_i-1}\left(\int_0^ts^{1/2-H_i}(t-s)^{-1/2-H_i}\left[\sigma^*(x_s)A^{-1}(x_s)\partial_{\theta_j}b(x_s,\theta_0)\right]_ids\right)\\
    &\hspace{30pt}\times\left(\int_0^ts^{1/2-H_i}(t-s)^{-1/2-H_i}\left[\sigma^*(x_s)A^{-1}(x_s)\partial_{\theta_k}b(x_s,\theta_0)\right]_ids\right)dt,~~~1\le j,k\le m,
  \end{aligned}
\end{equation*}
where
\begin{equation*}
  \begin{aligned}
   \gamma_i=\left(d_{H_i}\Gamma(1/2-H_i)\right)^{-2}.
 \end{aligned}
\end{equation*}
In order to guarantee the asymptotic properties of the estimator \eqref{likelihood}, we need to impose the identifibility condition. Let
\begin{equation*}
  \begin{aligned}
      \mathbb{Y}_{H,\varepsilon}(\theta):=&\varepsilon^2\left(\mathbb{L}_{H,\varepsilon}(\theta)-\mathbb{L}_{H,\varepsilon}(\theta_0)\right)
  \end{aligned}
\end{equation*}
and let $\mathbb{Y}_H$ be the expected limit of $\mathbb{Y}_{H,\varepsilon}$ which is defined by
\begin{equation*}
  \begin{aligned}
    \mathbb{Y}_H(\theta):=
       -\sum_{i=1}^r\frac{\gamma_i}{2}\int_0^T&t^{2H_i-1}
\\
       &\times\left\{\int_0^ts^{1/2-H_i}(t-s)^{-1/2-H_i}\left[\sigma^*(x_s)A^{-1}(x_s)\left(b(x_s,\theta)-b(x_s,\theta_0)\right)\right]^ids\right\}^2dt.
  \end{aligned}
\end{equation*}
We impose two following assumptions to $\Gamma_H$ and $\mathbb{Y}_H$.
  \begin{assumption}
    \label{(A3)}
    The matrix $\Gamma_H(\theta_0)$ is positive definite.
  \end{assumption}
  \begin{assumption}
    \label{(A4)}
    There exists a positive constant $\xi(\theta_0)>0$ and such that
    \begin{equation*}
      \mathbb{Y}_H(\theta)\le-\xi(\theta_0)|\theta-\theta_0|^2,
    \end{equation*}
    for every $\theta\in\Theta$.
  \end{assumption}
  The following theorem gives the asymptotic properties of the estimator $\hat{\theta}_\varepsilon$.
\begin{theorem}
  \label{main}
 Suppose that Assumptions \ref{(A1)}-\ref{A3} and \ref{(A3)}, \ref{(A4)} are fulfilled. Then the estimator $\hat{\theta}_\varepsilon$ satisfies that
 \begin{equation*}
   \varepsilon^{-1}(\hat{\theta}_\varepsilon-\theta_0)\xrightarrow{d}N(0,\Gamma_H(\theta_0)^{-1})
 \end{equation*}
 as $\varepsilon\rightarrow 0$. Moreover, we have
 \begin{equation*}
   E\left[f\left(\varepsilon^{-1}(\hat{\theta}_\varepsilon-\theta_0)\right)\right]\rightarrow E[f(\xi)]
 \end{equation*}
 as $\varepsilon\rightarrow 0$ for every continuous function $f$ of polynomial
  growth, where $\xi\sim N(0,\Gamma_H(\theta_0)^{-1})$.
\end{theorem}

\section{Proof of Theorem \ref{main}}
 We first establish a core lemma to guarantee the asymptotic properties of the estimator \eqref{likelihood} when $\varepsilon\rightarrow 0$.
 \begin{lemma}
   \label{s}
  Let Assumptions \ref{(A1)}, \ref{(A2)} be fulfilled. Then for every $p>0$ and $t\in[0,T]$, there exists a constant $C>0$ such that
     \begin{align}
       E|X_t^\varepsilon-x_t|^p&\le C\varepsilon^p
       \label{ds}.
     \end{align}
 \end{lemma}
 \begin{proof}
  By Assumption \ref{(A1)} and Proposition \ref{rie},  we have
 \begin{equation*}
   \begin{aligned}
     |X^\varepsilon_t-x_t|&\lesssim \int_0^t|X^\varepsilon_s-x_s|ds+\varepsilon\left|\int_0^t\sigma(X^\varepsilon_s)dB_s\right|\\
     &\lesssim\int_0^t|X^\varepsilon_s-x_s|ds+\varepsilon\left|\int_0^t\sigma(X^\varepsilon_s)dB_s-\sigma(X_0)B_t-\varepsilon\nabla_x\sigma(X_0)\sigma(X_0)\mathbb{B}_{0,t}\right|
     \\
     &\hspace{100pt}+\varepsilon\left|\sigma(X_0)B_t+\varepsilon\nabla_x\sigma(X_0)\sigma(X_0)\mathbb{B}_{0,t}\right|\\
     &\lesssim\int_0^t|X^\varepsilon_s-x_s|ds+\varepsilon\left(\|B\|_{\alpha,[0,T]}\|R^{\sigma(X^\varepsilon)}\|_{2\alpha,[0,T]}+\|\mathbb{B}\|_{2\alpha,[0,T]}\|\sigma(X^\varepsilon)\|_{\alpha,[0,T]}\right)
     t^{3\alpha}\\
     &\hspace{100pt}+\varepsilon\left|\sigma(X_0)B_t+\varepsilon\nabla_x\sigma(X_0)\sigma(X_0)\mathbb{B}_{0,t}\right|,
   \end{aligned}
 \end{equation*}
 for $1/3<\alpha<h$. By using Assumption \ref{(A2)} and Lemma \ref{HE},
 \begin{equation*}
   \begin{aligned}
        |X_t-x_t|&\lesssim\int_0^t|X_s-x_s|ds+\varepsilon\bigl(\|B\|_{\alpha,[0,T]}\|R^{\sigma(X)}\|_{2\alpha,[0,T]}+\|\mathbb{B}\|_{2\alpha,[0,T]}\|X\|_{\alpha,[0,T]}\\
        &\hspace{200pt}+\|B\|_{\alpha,[0,T]}+\|\mathbb{B}\|_{2\alpha,[0,T]}\bigl)\\
        &\lesssim\int_0^t|X_s-x_s|ds+\varepsilon\Biggl(\|B\|_{\alpha,[0,T]}\|(B,\mathbb{B})\|^{2/\alpha}_{\alpha,[0,T]}\\
        &\hspace{100pt}+\|\mathbb{B}\|_{2\alpha,[0,T]}\|(B,\mathbb{B})\|^{1/\alpha}_{\alpha,[0,T]}+\|B\|_{\alpha,[0,T]}+\|\mathbb{B}\|_{2\alpha,[0,T]}\Biggl).
   \end{aligned}
 \end{equation*}
 By the Grownwall's inequality,
 \begin{equation*}
   E|X_t-x_t|^p\lesssim\varepsilon^p E\left(\|B\|_{\alpha,[0,T]}\|(B,\mathbb{B})\|^{2/\alpha}_{\alpha,[0,T]}
   +\|\mathbb{B}\|_{2\alpha,[0,T]}\|(B,\mathbb{B})\|^{1/\alpha}_{\alpha,[0,T]}+\|B\|_{\alpha,[0,T]}+\|\mathbb{B}\|_{2\alpha,[0,T]}\right)^p,
 \end{equation*}
 and we complete the proof of Lemma \ref{s}.
 \end{proof}
 To show the asymptotic property of estimator $\hat{\theta}_\varepsilon$, we apply the polynomial type large deviation inequality investigated by Yoshida \cite{Yoshida}. Let $\mathbb{U}_\varepsilon(\theta_0):=\left\{u\in\mathbb{R}^m:\theta_0+\varepsilon u\in\Theta\right\}$ and define the random field $\mathbb{Z}_{H,\varepsilon}:\mathbb{U}_\varepsilon(\theta_0)\rightarrow\mathbb{R}_+$ by
 \begin{equation*}
   \mathbb{Z}_{H,\varepsilon}(u)=\exp\left\{\mathbb{L}_{H,\varepsilon}(\theta_0+\varepsilon u)-\mathbb{L}_{H,\varepsilon}(\theta_0)\right\},~~~u\in \mathbb{U}_\varepsilon(\theta_0).
 \end{equation*}
 Applying Taylor's formula, we have
 \begin{equation*}
   \log\mathbb{Z}_{H,\varepsilon}(u)=\varepsilon\nabla_\theta\mathbb{L}_{H,\varepsilon}(\theta_0)[u]-\frac{1}{2}u\Gamma_H(\theta_0)u^\top+R_\varepsilon(u),
 \end{equation*}
 where
 \begin{equation*}
   \begin{aligned}
     R_\varepsilon(u)&=\frac{1}{2}u\left(\varepsilon^2\nabla_\theta^2\mathbb{L}_{H,\varepsilon}(\theta_0)-(-\Gamma_H(\theta_0))\right)u^\top\\
     &~~~~~~+\frac{1}{2}\varepsilon^3\int_0^1(1-s)^2\nabla_\theta^3\mathbb{L}_{H,\varepsilon}(\theta_0+s\varepsilon u)[u,u,u]ds.
   \end{aligned}
 \end{equation*}
 and $\nabla_\theta^3\mathbb{L}_{H,\varepsilon}(\theta)[u,v,w]=\sum_{i,j,k}\partial_i\partial_j\partial_k\mathbb{L}_{H,\varepsilon}(\theta)u_iv_jw_k$.
 We prepare some lemmas for the sufficient conditions of polynomial type large deviation inequality in \cite{Yoshida}.
 \begin{lemma}
   \label{e}
   For every $p>0$,
   \begin{equation*}
     \sup_{0<\varepsilon<1}E\left[\left(\varepsilon^{-1}\left|\varepsilon^2\nabla_{\theta}^2\mathbb{L}_{H,\varepsilon}(\theta_0)-(-\Gamma_H(\theta_0))\right|\right)^p\right]<\infty.
   \end{equation*}
 \end{lemma}
 \begin{proof}
   By Remark \ref{dif}, it folows that
   \begin{equation}
     \label{Qd}
     \begin{aligned}
       \varepsilon^2\nabla_{\theta}^2\mathbb{L}_{H,\varepsilon}(\theta_0)-(-\Gamma_H(\theta_0))&=\varepsilon^2\sum_{i=1}^r\int_0^T\nabla_\theta^2 Q_{H_i,\theta_0}^\varepsilon(t)dW_t^i
       \\
       &\hspace{100pt}-\left(\varepsilon^2\sum_{i=1}^r\int_0^T\left(\nabla_\theta Q_{H_i,\theta_0}^\varepsilon(t)\right)^{\otimes2}dt
           -\Gamma_H(\theta_0)\right).
     \end{aligned}
   \end{equation}
   By Burkholder's and Minkowski's inequalities, Assumption \ref{(A1)}-\ref{A3} and Lemma \ref{HE}, the stochastic integral part of \eqref{Qd} is estimated as
   \begin{equation*}
     \begin{aligned}
       &E\left(\varepsilon^2\left|\int_0^T\partial_{\theta_j}\partial_{\theta_k} Q_{H_i,\theta_0}^\varepsilon(t)dW_t^i\right|\right)^p
       \lesssim \varepsilon^{2p}\left(\int_0^T\left \|\partial_{\theta_j}\partial_{\theta_k} Q_{H_i,\theta_0}^\varepsilon(t)\right\|_{L^p(\Omega)}^2dt\right)^{p/2}\\
       &\lesssim\varepsilon^p\left(\int_0^Tt^{2H_i-1}\left|\int_0^ts^{1/2-H_i}(t-s)^{-1/2-H_i}\left\|\partial_{\theta_j}\partial_{\theta_k}\left[\sigma^*(X_s^\varepsilon)A^{-1}(X_s^\varepsilon)b(X_s^\varepsilon,\theta_0)\right]^i\right\|_{L^p(\Omega)}ds\right|^2dt\right)^{p/2}\\
       &\lesssim\varepsilon^p\sup_{0\le s\le T}\left\|1+|X_s^\varepsilon|^N\right\|^p_{L^p(\Omega)}\left(\int_0^Tt^{1-2H_i}dt\right)^{p/2}\lesssim\varepsilon^p,
     \end{aligned}
   \end{equation*}
   for every $1\le i\le r$ and $1\le j,k\le m$. We shall estimate the second part of \eqref{Qd}. For every $1\le j,k\le m$,
   \begin{equation*}
     \begin{aligned}
       &\varepsilon^2\int_0^T\left(\partial_{\theta_j} Q_{H,\theta_0}^\varepsilon(t)\right)^*\partial_{\theta_k} Q_{H,\theta_0}^\varepsilon(t)dt-\Gamma_H^{j,k}(\theta_0)
       \\
       &=\int_0^T\left[\varepsilon\left(\partial_{\theta_j} Q_{H,\theta_0}^\varepsilon(t)\right)^*\left(\varepsilon\partial_{\theta_k} Q_{H,\theta_0}^\varepsilon(t)-\partial_{\theta_k} Q_{H,\theta_0}(t)\right)
       +\left(\varepsilon\partial_{\theta_j} Q_{H,\theta_0}^\varepsilon(t)-\partial_{\theta_j} Q_{H,\theta_0}(t)\right)^*\partial_{\theta_k} Q_{H,\theta_0}(t)\right]dt.
     \end{aligned}
   \end{equation*}
   By H\"older's and Minkowski's inequalities, the mean value theorem and Lemma \ref{dif}, we can show that
   \begin{equation*}
     \begin{aligned}
       &E\left|\int_0^T\varepsilon\left(\partial_{\theta_j} Q_{\theta_0}^\varepsilon(t)\right)^*\left(\varepsilon\partial_{\theta_k} Q_{\theta_0}^\varepsilon(t)-\partial_{\theta_k} Q_{\theta_0}(t)\right)dt\right|^p
       \\
       &=E\Biggl|\sum_{i=1}^r\gamma_i\int_0^Tt^{2H_i-1}\Bigl(\int_0^t s^{1/2-H_i}(t-s)^{-1/2-H_i}
       \\
       &\hspace{100pt}\times\bigl(\sigma^*(X_s^\varepsilon)A^{-1}(X_s^\varepsilon)\partial_{\theta_k}b(X_s^\varepsilon,\theta_0)-\sigma^*(x_s)A^{-1}(x_s)\partial_{\theta_k}b(x_s,\theta_0)\bigl)_ids\Bigl)\\
       &\hspace{120pt}\cdot\left(\int_0^t s^{1/2-H_i}(t-s)^{-1/2-H_i}\left(\sigma^*(X_s^\varepsilon)A^{-1}(X_s^\varepsilon)\partial_{\theta_j}b(X_s^\varepsilon,\theta_0)\right)_ids\right)dt\Biggl|^p\\
       &\lesssim\Biggl|\sum_{i=1}^r\gamma_i\int_0^Tt^{2H_i-1}\Bigl(\int_0^t s^{1/2-H_i}(t-s)^{-1/2-H_i}\\
       &\hspace{100pt}\left\|\sigma^*(X_s^\varepsilon)A^{-1}(X_s^\varepsilon)\partial_{\theta_k}b(X_s^\varepsilon,\theta_0)-\sigma^*(x_s)A^{-1}(x_s)\partial_{\theta_k}b(x_s,\theta_0)\right\|_{L^{2p}(\Omega)}ds\Bigl)\\
       &\hspace{100pt}\times\left(\int_0^t s^{1/2-H_i}(t-s)^{-1/2-H_i}\left\|\sigma^*(X_s^\varepsilon)A^{-1}(X_s^\varepsilon)\partial_{\theta_j}b(X_s^\varepsilon,\theta_0)\right\|_{L^{2p}(\Omega)}ds\right)dt\Biggl|^p\\
       &\lesssim\sup_{0\le s\le T}\left\|\left(1+|X_s^\varepsilon|^N+|x_s|^N\right)\left|X_s^\varepsilon-x_s\right|\right\|_{L^{2p}(\Omega)}^p\left\|\left(1+|X_s^\varepsilon|^N\right)\right\|_{L^{2p}(\Omega)}^p\left(\sum_{i=1}^r\gamma_i\int_0^Tt^{1-2H_i}dt\right)^p
       \lesssim\varepsilon^p.
     \end{aligned}
   \end{equation*}
   In the same way, we can show that
   \begin{equation*}
     E\left|\int_0^T
     \left(\varepsilon\partial_{\theta_j} Q_{H,\theta_0}^\varepsilon(t)-\partial_{\theta_j} Q_{H,\theta_0}(t)\right)^*\partial_{\theta_k} Q_{H,\theta_0}(t)dt\right|^p\lesssim\varepsilon^p.
   \end{equation*}
 Therefore
 \begin{equation*}
   \sup_{0<\varepsilon<1}E\left[\left(\varepsilon^{-1}\left|\varepsilon^2\nabla_{\theta}^2\mathbb{L}_{H,\varepsilon}(\theta_0)-(-\Gamma_H(\theta_0))\right|\right)^p\right]<\infty,
 \end{equation*}
 and we complete the proof of Lemma \ref{e}.
\end{proof}
\begin{lemma}
  \label{3}
    For every $p\ge 2$,
  \begin{equation*}
    \sup_{0<\varepsilon<1}E\left[\sup_{\theta\in\Theta}\left|\varepsilon^2\nabla_{\theta}^3\mathbb{L}_{H,\varepsilon}(\theta)\right|^p\right]<\infty.
  \end{equation*}
\end{lemma}
\begin{proof}
  By Sobolev's inequality, for every $p>d$
  \begin{equation*}
    \sup_{\theta\in\Theta}\left|\nabla_{\theta}^3\mathbb{L}_{H,\varepsilon}(\theta_0)\right|^p\lesssim\int_\Theta\left(\left|\nabla_{\theta}^3\mathbb{L}_{H,\varepsilon}(\theta)\right|^p+\left|\nabla_{\theta}^4\mathbb{L}_{H,\varepsilon}(\theta)\right|^p\right)d\theta.
  \end{equation*}
  By the same argument as in the proof of Lemma \ref{e}, we can show that
  \begin{equation*}
    \sup_{0<\varepsilon<1}\varepsilon^{2p}E\left[\int_\Theta\left(\left|\nabla_{\theta}^3\mathbb{L}_{H,\varepsilon}(\theta)\right|^p+\left|\nabla_{\theta}^4\mathbb{L}_{H,\varepsilon}(\theta)\right|^p\right)d\theta\right]<\infty.
  \end{equation*}
\end{proof}
We can check the following lemma holds true. The proof is similar to the one for Lemmas \ref{e} and \ref{3}.
\begin{lemma}
  \label{4}
  For every $p\ge 2$,
  \begin{equation*}
    \sup_{0<\varepsilon<1}E|\varepsilon\partial_\theta\mathbb{L}_{H,\varepsilon}(\theta_0)|^p<\infty,
  \end{equation*}
  and
  \begin{equation*}
    \sup_{0<\varepsilon<1}E\left(\sup_{\theta\in\Theta}\varepsilon\left|\mathbb{Y}_{H,\varepsilon}(\theta)-\mathbb{Y}_H(\theta)\right|\right)^p<\infty
  \end{equation*}
\end{lemma}
From Lemma \ref{4} and the proof of Lemma \ref{e}, we obtain that
\begin{equation}
\label{n}
\varepsilon\partial_\theta\mathbb{L}_{H,\varepsilon}(\theta_0)\xrightarrow{d}N(0,\Gamma_H(\theta_0)),
\end{equation}
as $\varepsilon\rightarrow 0$. Moreover, Lemmas \ref{e}, \ref{3} and the convergence \eqref{n} give the local asymptotic normality of $Z_{H,\varepsilon}(u)$:
\begin{equation*}
  Z_{H,\epsilon}(u)\xrightarrow{d}Z_{H}(u):=\exp\left(\Delta_H(\theta_0)u-\frac{1}{2}\Gamma_H(\theta_0)[u,u]\right),
\end{equation*}
where $\Delta_H(\theta_0)\sim N(0,\Gamma_H(\theta_0))$. In the sequel, we check conditions theorem 3, (c) in \cite{Yoshida}. Lemma \ref{e} and \ref{3} give [A1$''$] and [A4$'$] in \cite{Yoshida}
with $\beta_1\thickapprox1/4,~\rho_1,\rho_2,\beta,\beta_2\thickapprox0$. The conditions [B1] and [B2] in \cite{Yoshida} follow from Assumptions \ref{(A3)}, \ref{(A4)}. Moreover, the condition [A6] is true from Lemma \ref{4}. Now Theorem 3 in \cite{Yoshida} yields the inequality
\begin{equation}
  \sup_{0<\varepsilon<1}P\left[\sup_{|u|\ge r}\mathbb{Z}_{H,\varepsilon}(u)\ge e^{-r}\right]\lesssim r^{-L}
\end{equation}
hold for any $r>0$ and $L>0$. Since $u_\varepsilon:=\varepsilon^{-1}(\theta_\varepsilon-\theta_0)$ maximizes the random field $Z_{H,\varepsilon}$, the sequence $\{f(u_\varepsilon)\}_{\varepsilon}$ is uniformly integrable for every continuous function $f$ such that for every $x\in\mathbb{R}$, $f(x)\lesssim 1+|x|^N$ for some $N>0$. Indeed,
\begin{equation*}
  \begin{aligned}
    \sup_{0<\varepsilon<1}P\left(|u_\varepsilon|\ge r\right)&\le\sup_{0<\varepsilon<1}P\left(\sup_{|u|\ge r}\mathbb{Z}_{H,\varepsilon}(u)\ge\mathbb{Z}_{H,\varepsilon}(0)\right)
    \lesssim r^{-L},
  \end{aligned}
\end{equation*}
for every $r>0$ and $L>0$. Thus
\begin{equation*}
  \begin{aligned}
     \sup_{0<\varepsilon<1}E[|f(u_\varepsilon)|]&\lesssim 1+\int_0^\infty\sup_{0<\varepsilon<1} P\left(|u_\varepsilon|>r^{1/N}\right)dr<\infty.
  \end{aligned}
\end{equation*}
  Let $B(R):=\left\{u\in\mathbb{R}^m; |u|\le R\right\}$. In the sequel, we prove that
\begin{equation}
  \label{f}
  \log\mathbb{Z_{H,\varepsilon}}\xrightarrow{d}\log\mathbb{Z}_{H,\varepsilon}~~~{\rm in}~C(B(R)),
\end{equation}
as $\varepsilon\rightarrow 0$. If we can show the convergence \eqref{f}, we obtain the asymptotic normality:
\begin{equation*}
   \varepsilon^{-1}(\theta_\varepsilon-\theta_0)\xrightarrow{d}N(0,\Gamma_H(\theta_0)^{-1})
\end{equation*}
as $\varepsilon\rightarrow 0$ by Theorem 5 in Yoshida \cite{Yoshida}. Due to linearity in $u$ of the weak convergence term $\varepsilon\nabla_\theta\mathbb{L}_{H,\varepsilon}(\theta_0)[u]$, the convergence of finite--dimensional distribution holds true. It remains to show the tightness of the family $\left\{\log\mathbb{Z}_{H,\varepsilon}(u)\right\}_{u\in B(R)}$. By the Kolmogorov tightness criterion, it suffices to show that
for every $R>0$ there exists a constant $p>0,~\gamma>d$ and $C>0$ such that
\begin{equation}
  \label{tight}
  E\left|\log\mathbb{Z}_{H,\varepsilon}(u_1)-\log\mathbb{Z}_{H,\varepsilon}(u_2)\right|^p\le C|u_1-u_2|^\gamma,
\end{equation}
for $u_1,u_2\in B(R).$ For a number $p>0$ large enough, the inequality \eqref{tight} is shown easily by Lemmas \ref{e}, \ref{3} and \ref{4}. Therefore, we complete the proof of Theorem \ref{main}.

\section*{Declaration}
\textbf{Conflict of interest}\\
The author declares that there is no conflict of interest.
\bibliographystyle{unsrtnat}

\end{document}